\documentclass[12pt]{amsart}
\usepackage{amssymb,amsfonts,latexsym,amscd}

\usepackage[all,cmtip,2cell]{xy}
\UseAllTwocells
\usepackage{verbatim}

\newtheorem{theorem}{Theorem}[section]
\newtheorem{lemma}[theorem]{Lemma}
\newtheorem{proposition}[theorem]{Proposition}
\newtheorem{corollary}[theorem]{Corollary}
\theoremstyle{definition}
\newtheorem{definition}[theorem]{Definition}

\newtheorem{example}[theorem]{Example}
\newtheorem{question}[theorem]{Question}

\newtheorem{remark}[theorem]{Remark}

\newcommand{\id}{\text{id}}

\newcommand{\Ker}{{\rm Ker }}

\newcommand{\End}{{\rm End}}

\newcommand{\FPdim}{\text{\rm FPdim}}
\newcommand{\Ext}{\text{\rm Ext}}

\newcommand{\Hom}{{\rm Hom}}

\newcommand{\Rep}{\text{Rep}}

\newcommand{\Vect}{{\rm Vec}}

\newcommand{\A}{\mathcal{A}}
\newcommand{\B}{\mathcal{B}}
\newcommand{\C}{\mathcal{C}}

\newcommand{\M}{\mathcal{M}}
\newcommand{\N}{\mathcal{N}}

\newcommand{\ben}{\begin{enumerate}}
\newcommand{\een}{\end{enumerate}}

\begin{document}

\title[Exact sequences of tensor categories] {Exact sequences of tensor categories with respect to a module category}

\author{Pavel Etingof}
\address{Department of Mathematics, Massachusetts Institute of Technology,
Cambridge, MA 02139, USA} \email{etingof@math.mit.edu}

\author{Shlomo Gelaki}
\address{Department of Mathematics, Technion-Israel Institute of
Technology, Haifa 32000, Israel} \email{gelaki@math.technion.ac.il}

\date{\today}


\begin{abstract}
We generalize the definition of an exact sequence of tensor categories 
due to Brugui\`eres and Natale, and introduce a new notion of an exact sequence of 
(finite) tensor categories with respect to a module category. 
We give three definitions of this notion and show their equivalence. 
In particular, the Deligne tensor product of tensor categories 
gives rise to an exact sequence in our sense. We also show 
that the dual to an exact sequence in our sense is again an exact sequence. 
This generalizes the corresponding statement for exact sequences of Hopf algebras.   
Finally, we show that the middle term of an exact sequence 
is semisimple if so are the other two terms. 
\end{abstract}

\maketitle

\section{introduction}

The theory of exact sequences of tensor categories was developed by A. Brugui\`eres and S. Natale
(\cite{BN, BN1}) as a categorical generalization of the theory of exact sequences of Hopf algebras (\cite{Mo,R,Sch}).
Namely, let $B$ be a Hopf algebra (which for simplicity we will assume finite dimensional).
Recall that a Hopf subalgebra $A\subseteq B$  is {\it normal} if it is invariant under the adjoint action of $B$ on itself 
(this generalizes the notion of a normal subgroup). 
In this case, it is easy to show that $A_+B$ (where $A_+:={\rm Ker}(\varepsilon)|_A$ is the augmentation ideal in $A$) 
is a two-sided ideal in $B$, hence a Hopf ideal, and 
thus the quotient $C:=B/A_+B$ is a Hopf algebra. 
In this situation, one says that one has a (short) 
{\it exact sequence of Hopf algebras}\footnote{Note that such a sequence is not a short exact sequence of vector spaces: the composition is not zero but rather is the augmentation map $\varepsilon$, and  $\dim (B)$ equals
$\dim (A)\dim (C)$ rather than $\dim (A)+\dim (C)$.} 
\begin{equation}\label{esha} 
A\to B\to C.
\end{equation}

The sequence \eqref{esha} defines
a sequence of tensor functors \begin{equation}\label{exactseq}
\mathcal{A}\xrightarrow{\iota}\mathcal{B}\xrightarrow{F}\mathcal{C}
\end{equation}
between the corresponding tensor categories of finite dimensional comodules,
$\mathcal{A}:=A-{\rm comod}$, $\mathcal{B}:=B-{\rm comod}$, $\mathcal{C}:=C-{\rm comod}$. 
This sequence has the following categorical properties: 

(i) The functor $F$ is {\it surjective} (or {\it dominant}), 
i.e., any object $Y$ of $\mathcal{C}$ is a subquotient 
(equivalently, a subobject, a quotient) of $F(X)$ for 
some $X\in \mathcal{B}$. 

(ii) The functor $\iota$ is {\it injective}, i.e., is a fully faithful embedding. 

(iii) The kernel of $F$ (i.e., the category of objects $X$ such that $F(X)$ is 
trivial) coincides with the image of $\iota$. 

(iv) The functor $F$ is {\it normal}, i.e., for any $X\in \mathcal{B}$ 
there exists a subobject $X_0\subseteq X$ such that $F(X_0)$ 
is the largest trivial subobject of $F(X)$. 

Brugui\`eres and Natale called a sequence \eqref{exactseq} satisfying conditions (i)-(iv) 
{\it an exact sequence of tensor categories}. This is a very nice and natural definition, but 
its drawback is that conditions (ii), (iii) force the category $\mathcal{A}$ 
to have a tensor functor to ${\rm Vec}$ (namely, $F\circ \iota$), i.e., to be the category of comodules
over a Hopf algebra. In particular, the Deligne tensor product  
$\mathcal{B}:=\mathcal{C}\boxtimes \mathcal{A}$ of finite 
tensor categories $\mathcal{C}$, $\mathcal{A}$  
is not included in an exact sequence \eqref{exactseq}
unless $\mathcal{A}$ admits a tensor functor to ${\rm Vec}$
(i.e., a fiber functor).

The goal of this paper is to generalize the definition of \cite{BN} further to 
eliminate this drawback, and in particular to include the example
$\mathcal{B}:=\mathcal{A}\boxtimes \mathcal{C}$
for any finite tensor categories $\mathcal{A}$, $\mathcal{C}$. 
Namely, we do so by replacing the category ${\rm Vec}$ 
by the category ${\rm End}(\mathcal{M})$ of right exact endofunctors of 
an indecomposable exact $\mathcal{A}$-module category $\mathcal{M}$, 
and define the notion of an exact sequence ``with respect to $\mathcal{M}$",
which is a sequence of tensor functors of the form 
$$
\mathcal{A}\xrightarrow{\iota} \mathcal{B}\xrightarrow{F} \mathcal{C}\boxtimes \End(\M),$$
such that $\iota$ is injective, $F$ is surjective, $\A={\rm Ker}(F)$ (the subcategory of $X\in \B$ such that $F(X)\in \End(\M)$)\footnote{Here and below, for brevity we write $\A$ for $\iota(\A)$.}, 
and $F$ is normal (i.e., for any $X\in \mathcal{B}$ 
there exists a subobject $X_0\subseteq X$ such that $F(X_0)$ 
is the largest subobject of $F(X)$ contained in $\End(\M)$). 
Moreover, we show that the dual of an exact sequence is again an exact sequence. 
This is an important feature of the original setting of exact sequences of Hopf algebras,
which breaks down in the generalization of \cite{BN}, but is restored in our more general setting. 
We also present another definition of exact sequences with respect to a module category, and
show that the two definitions are equivalent. We then generalize a number of results of \cite{BN} 
to our setting; in particular, we show that for any exact sequence, 
${\rm FPdim}(\mathcal{B})={\rm FPdim}(\mathcal{A}){\rm FPdim}(\mathcal{C})$, and that this property in fact 
characterizes exact sequences (provided that $\iota$ is injective, $F$ is surjective, and $\A\subseteq {\rm Ker}(F)$.
Finally, we show that if in such an exact sequence, $\A$ and $\C$ are semisimple, then so is $\B$. 

The structure of the paper is as follows. Section 2 contains preliminaries. 
In particular, in this section we generalize the theory of surjective tensor functors 
from \cite[Section 2]{EO} to the setting involving module categories. 
In Section 3 we give a definition of an exact sequence of finite tensor categories 
with respect to a module category. We also give two other definitions and show that all three definitions are 
equivalent, and prove the semisimplicity of $\B$ given the semisimplicity of $\A$, $\C$. 
Finally, in Section 4, we discuss dualization of exact sequences and give some examples. 

\begin{remark} 1. For simplicity we present our definitions and results for {\em finite} tensor categories, but many of our statements 
extend with appropriate changes to the case of arbitrary tensor and multitensor categories. 

2. In the special case of fusion categories (i.e., when all the categories involved are semisimple), 
the proofs of our results are much simpler, since $\End(\M)$ is a rigid category, 
all objects are projective, all module categories and functors are exact, etc. 
\end{remark} 

{\bf Acknowledgements.} We are very grateful to Dennis Gaitsgory, who suggested an alternative definition 
of exact sequences given in Theorem 
\ref{ges}(2),(3) (in the special case of \cite{BN}, i.e., Corollary \ref{bnes}). 
We also thank D. Nikshych and V. Ostrik for useful discussions.
In particular, we are very grateful to V. Ostrik for providing a proof of Proposition \ref{algint1} in the semisimple case
(our proof is a generalization of his). The work of P.E. was  partially supported by the NSF grant DMS-1000113.
The work of S.G. was partially supported by the Israel Science Foundation (grant no. 561/12). 

\section{Preliminaries}

\subsection{A linear algebra lemma}

The following lemma is well known. 

\begin{lemma}\label{lal} 
Let $A$ be a square matrix with positive entries, and $B$ be a square matrix of the same size with $0\le b_{ij}\le a_{ij}$, 
such that $B\ne A$. Then the largest eigenvalues $\lambda(A),\lambda(B)$ of $A,B$ satisfy $\lambda(B)<\lambda(A)$. 
\end{lemma}

\begin{proof} 
Let $v$ be a positive column eigenvector of $A$ (defined uniquely up to scale), i.e., $Av=\lambda(A)v$. 
Let $w$ be a nonnegative row eigenvector of $B$, i.e., $wB=\lambda(B)w$. 
(Such eigenvectors exist by the Frobenius-Perron theorem.)
Then 
$$
\lambda(B)w\cdot v=wBv\le wAv=\lambda(A)w\cdot v, 
$$
which implies that $\lambda(B)\le \lambda(A)$ (as $w\ne 0$ and hence $w\cdot v>0$). 
Moreover, if $B$ has strictly positive entries, 
then so does the vector $w$, so $wBv<wAv$, and hence $\lambda(B)<\lambda(A)$. 

Now note that $\frac{1}{2}(A+B)$ has strictly positive entries. 
Thus, applying the above to the pairs of matrices $\frac{1}{2}(A+B),B$ 
and $A,\frac{1}{2}(A+B)$, we get 
$$
\lambda(B)\le \lambda\left(\frac{1}{2}(A+B)\right)<\lambda(A),
$$
as claimed. 
\end{proof} 

\subsection{Surjective functors} 
Throughout the paper, we let $k$ be an algebraically closed field. 
Recall that an abelian category over $k$ is finite if it is equivalent to the category of modules 
over a finite dimensional $k$-algebra. 
Let $F: \M\to \N$ be an exact functor between finite abelian categories over $k$.

\begin{definition} We say that $F$ is {\it surjective}, or {\it dominant}, if 
any $Y\in \N$ is a subquotient of $F(X)$ for some $X\in \M$. 
\end{definition}

The following lemma is standard. 

\begin{lemma}\label{eqcon}  
Assume that $F:\M\to \N$ is as above, and that in $\N$, projectives and injectives coincide. 
Then the following conditions are equivalent: 

(i) Any $Y\in \N$ 
is a subobject of $F(X)$ for some $X\in \M$. 

(ii) Any $Y\in \N$ 
is a quotient of $F(X)$ for some $X\in \M$. 

(iii) Any $Y\in \N$ 
is a subquotient of $F(X)$ for some $X\in \M$ (i.e., $F$ is surjective).  
\end{lemma} 

\begin{proof} Let us show that (iii) implies (i). Let $I(Y)$ be the injective hull of $Y$. 
Then $I(Y)$ is a subquotient of $F(X)$ for some $X\in \M$. Thus, $I(Y)\subseteq Z$, where 
$Z$ is a quotient of $F(X)$. But since $I(Y)$ is injective, $Z=I(Y)\oplus Z'$, so we 
see that $I(Y)$ is a quotient of $F(X)$. But since projectives and injectives coincide in $\N$,  
$I(Y)$ is also projective, hence $F(X)=I(Y)\oplus T$ for some $T$. Hence, $Y\subseteq F(X)$. 

Similarly, to prove that (iii) implies (ii), let $P(Y)$ be the projective cover of $Y$, 
and let $X$ be such that $P(Y)$ is a subquotient of $F(X)$. Since $P(Y)$ is both projective and injective, 
it is a direct summand in $F(X)$, so $Y$ is a quotient of $F(X)$. 

The rest is obvious.  
\end{proof} 

\subsection{Dual module categories}
Let $\mathcal{A}$ be a finite tensor category over $k$ (see \cite{EO}).
Let $\M$ be a left $\mathcal{A}$-module category. 
Then the opposite category $\M^{\rm op}$ has two (in general, different) 
structures of a right $\A$-module category. Namely, we can define 
the category $\M^\vee$ with the right $\A$-action defined by 
$M\otimes X:={}^*X\otimes M$, and the category 
${}^\vee\M$ with the right $\A$-action defined by 
$M\otimes X:=X^*\otimes M$. Similarly, if 
$\N$ is a right $\A$-module category, then 
we can define two left $\A$-module structures on $\N^{\rm op}$ as 
follows: the category $\N^\vee$ with $X\otimes M:=M\otimes {}^*X$, and 
the category ${}^\vee\N$ with $X\otimes M:=M\otimes X^*$. 
It is easy to see that ${}^\vee(\M^\vee)=\M$ for any 
right or left $\A$-module category $\M$
(cf. \cite[3.4.2]{DSS} and \cite[7.1]{EGNO}). Also note that 
if $A\in \A$ is an algebra, and $\M$ is 
the category of right $A$-modules in $\A$ 
(which is therefore a left $\A$-module category), 
then the category of left $A$-modules 
is naturally equivalent to $\M^\vee$ as a right $\A$-module category. 

\subsection{Tensor products of module categories} 
Let $\A$, $\M$, $\N$ be as in the previous subsection. 
Consider the Tambara (or balanced) tensor product 
$\mathcal{N} \boxtimes_{\mathcal{A}}\mathcal{M}$ (see \cite{DSS,DSS2,ENO,Ta}).
Namely, if $A_1,A_2$ are algebras in $\A$ such that $\M={\rm mod}-A_1$ and 
$\N=A_2-{\rm mod}$, then $\mathcal{N} \boxtimes_{\mathcal{A}}\mathcal{M}$
is the category of $(A_2,A_1)$-bimodules in $\A$, 
which can also be described as the category 
of left $A_2$-modules in $\M$, or the category 
of right $A_1$-modules in $\N$ (see \cite{DSS2} and \cite[7.8]{EGNO}). 

By \cite[Corollary 3.4.11]{DSS}, one has natural equivalences
\begin{equation}\label{tenprod}
\mathcal{N} \boxtimes_{\mathcal{A}}\mathcal{M}\cong {\rm Fun}_{\mathcal{A}}(\mathcal{N}^\vee,\mathcal{M})\cong {\rm Fun}_{\mathcal{A}}({}^\vee\mathcal{M},\mathcal{N}).
\end{equation}

\subsection{Exact module categories}
Recall from \cite[Section 3]{EO} that a left $\A$-module category $\M$ is said to be {\em indecomposable} if it is not a direct sum of
two nonzero module categories, and is called {\em exact} if $P\otimes M$ is projective for any projective $P\in \A$ and any $M\in \M$. 
The same definition applies to right module categories. 

Recall (\cite[Corollary 3.6]{EO}) that, similarly to finite tensor categories, exact module categories are quasi-Frobenius,
i.e., projectives and injectives coincide in them. 

Let $\mathcal{M}$ be an indecomposable 
exact $\mathcal{A}$-module category\footnote{Here and below, by ``a module category'' we will mean a left module category, unless otherwise specified.} 
(such a category is always finite \cite[Lemma 3.4]{EO}). 
Let $\End(\M)$ be the abelian category of {\em right} exact endofunctors of
$\mathcal{M}$, and let $\mathcal{A}_{\mathcal{M}}^*:=
\End_{\mathcal{A}}(\mathcal{M})$ be the dual category, i.e., the category of
$\mathcal{A}-$linear exact endofunctors of $\mathcal{M}$. Recall that
composition of functors turns $\End(\M)$ into a monoidal
category,\footnote{Note that if $\M$ is not semisimple then the category $\End(\M)$ is not rigid, and its tensor product is not exact.} 
and $\mathcal{A}_{\mathcal{M}}^*$ into a finite tensor
category. 

By (\ref{tenprod}), we have $\mathcal{A}_{\mathcal{M}}^*={}^\vee \mathcal{M}\boxtimes_{\mathcal{A}}\mathcal{M}$. Hence, taking duals, we get 
$\mathcal{A}_{\mathcal{M}}^{*\rm op}=\mathcal{M}^\vee\boxtimes_{\mathcal{A}}\mathcal{M}$. Also, define a functor 
$$G: \M\boxtimes \M^\vee\to \End(\M),\,\,\,M_1\boxtimes M_2\mapsto
\Hom_\M(?,M_2)^*\boxtimes M_1.$$ Then it is easy to check 
that $G$ is an equivalence of $\A$-bimodule categories (this is, in fact, a special case of \cite[Corollary 3.4.11]{DSS}).  

\begin{lemma}\label{critex} Let $\M$ be a module category over a finite tensor category $\A$. Suppose that there exists a nonzero object $X\in \A$ such that 
$X\otimes M$ is projective for any $M\in \M$. Then $\M$ is exact. 
\end{lemma} 

\begin{proof}  
Let $Z\in \A$ be a simple object. Pick a simple object $Y\in \A$ such that $Y$ is a composition factor 
of $Z\otimes X^*$. Let $P(Y)$ be the projective cover of $Y$. Then $\Hom(P(Y),Z\otimes X^*)=\Hom(P(Y)\otimes X,Z)\ne 0$. 
Since $P(Y)\otimes X$ is projective, we have $P(Y)\otimes X=P(Z)\oplus S$ for some $S$. Now, for any $M\in \M$, 
$X\otimes M$ is projective, so $P(Y)\otimes X\otimes M$ is projective. 
Hence $P(Z)\otimes M$ is projective, as a direct summand of a projective object. 
This implies the lemma. 
\end{proof} 

\begin{corollary}\label{subcaexact} If $\A\subseteq \B$ are finite tensor categories, and 
$\M$ is a $\B$-module category which is exact as an $\A$-module category, then it is exact as a $\B$-module category. 
\end{corollary} 

\begin{proof} This follows immediately from Lemma \ref{critex}. 
\end{proof} 

\begin{proposition}\label{critex2} Let $\A$ be a finite tensor category, and $\M$ be an indecomposable exact module category over $\A$. 
Then 

(i) For every nonzero object $M\in \M$, the functor $F_M:\A\to \M$ given by $F_M(X):=X\otimes M$ is surjective.   

(ii) The action functor $F: \A\to \End(\M)$ given by $F(X)=X\otimes ?$ is surjective. 
\end{proposition} 

\begin{proof}
Let $M_j$ be the simple objects in $\M$, and $R_j$ be their projective covers.
Let $P_i$ be the indecomposable projective objects in $\A$, and let $P:=\oplus_i P_i$. Then $P_i\otimes M_r=\oplus_j b_{ir}^j R_j$, $b_{ir}^j\in \Bbb Z_+$ (as $\M$ is exact). 
This means that $F(P)=\oplus_{j,r}(\sum_i b_{ir}^j)R_j\boxtimes R_r^\vee\in \M\boxtimes \M^\vee=\End(\M)$, where 
$R_r^\vee$ is the projective cover of $M_r$ in $\M^\vee$.
Since for any $j,r$, $\underline{\Hom}(M_j,M_r)\ne 0$, we see that for any $j,r$, $\sum_i b_{ir}^j>0$. 
This means that for a sufficiently large $m$, any object of $\End(\M)$ is a quotient of $F(mP)$, 
and for any $M\ne 0$, any object of $\M$ is a quotient of $F_M(mP)$, proving both (i) and (ii).   
\end{proof} 

\subsection{Regular objects}
For a finite tensor category $\A$, let ${\rm Gr}(\A)$ be the Grothendieck ring of $\A$, and $K_0(\A)$ be
the group of isomorphism classes of projective objects in $\A$ (it is a bimodule over ${\rm Gr}(\A)$). We have a natural homomorphism $\tau_\A: K_0(\A)\to {\rm Gr}(\A)$
(in general, neither surjective nor injective). 
Recall \cite[Subsection 2.4]{EO} that we have a character $\FPdim: {\rm Gr}(\A)\to \Bbb R$, attaching to $X\in \A$ the Frobenius-Perron dimension of $X$. 
Recall also that we have a virtual projective object $R_\A\in K_0(\A)\otimes_{\Bbb Z}\Bbb R$, such that 
$XR_\A=R_\A X=\FPdim(X) R_\A$ for all $X\in {\rm Gr}(\A)$. Namely, we have 
$R_\A=\sum_i \FPdim(X_i)P_i$, where $X_i$ are the simple objects of $\A$, and $P_i$ 
are the projective covers of $X_i$. Using the homomorphism $\tau_\A$, we may also regard 
$R_\A$ as an element of ${\rm Gr}(\A)\otimes_{\Bbb Z}\Bbb R$ (in other words, for brevity we will write $R_\A$ 
instead of $\tau_\A(R_\A)$). Following \cite[Subsection 2.4]{EO}, we set $\FPdim(\A):=\FPdim(R_\A)$. 
 
Also, if $\M$ is an indecomposable exact $\A$-module category, let 
${\rm Gr}(\M)$ be the Grothendieck group of $\M$. 
Let $\lbrace{ M_j\rbrace}$ be the basis of simple objects of $\M$. 
Let $K_0(\M)$ be the group of isomorphism classes of projective objects 
in $\M$ (it is a module over ${\rm Gr}(\A)$). As before, we have a natural map $\tau_\M: K_0(\M)\to {\rm Gr}(\M)$. 

It follows from the Frobenius-Perron theorem that there is a unique 
up to scaling element $R_\M\in K_0(\M)\otimes_{\Bbb Z}\Bbb R$ such that for all $X\in {\rm Gr}(\A)$, 
$XR_\M=\FPdim(X)R_\M$. Namely, we have
$R_\M=\sum_j \FPdim(M_j)R_j$, where $R_j$ are the projective covers of $M_j$. 
The numbers $\FPdim(M_j)$ are defined uniquely up to scaling by the 
property $$\FPdim(X\otimes M)=\FPdim(X)\FPdim(M),\,\,\,X\in \A,\, M\in \M,$$ and it is convenient to normalize them in such a way that $\FPdim(R_\M)=\FPdim(\A)$, which we
will do from now on (see \cite[Proposition 3.4.4, Exercise 7.16.8]{EGNO}, and \cite[Subsection 2.2]{ENO2} for the semisimple case). 
It is clear that $R_\A M_j=\FPdim(M_j)R_\M$. Note that using $\tau_\M$, we may view $R_\M$ also as an element of ${\rm Gr}(\M)$. 

Let $F: \A\to \End(\M)$ be the action functor. Then it is straightforward to verify that under the identification $\End(\M)\cong \M\boxtimes \M^\vee$ (see Subsection 2.5), 
we have $F(R_\A)=R_\M\boxtimes R_{\M^\vee}$. 

\subsection{Frobenius-Perron dimensions}

Let $\A$ be a finite tensor category, $\M$ an exact $\A$-module category, and $A$ an algebra in $\A$ such that $\M={\rm mod}-A$. 
In this case, we have two notions of Frobenius-Perron dimensions of an object $M\in \M$: 
the usual Frobenius-Perron dimension $\FPdim_\M(M)$ defined above, and also the Frobenius-Perron dimension 
of $M$ as a right $A$-module (viewed as an object of $\A$), which we will denote by 
$\FPdim_\A(M)$. We have 
$$
\FPdim_\A(X\otimes M)=\FPdim(X)\FPdim_\A(M),\ X\in \A ,
$$ 
so by the Frobenius-Perron theorem, there exists $\lambda>0$ such that 
$\FPdim_\A(M)=\lambda \FPdim_\M(M)$ for all $M\in \M$. 

Similarly, since $\A_\M^{*\rm op}=A-{\rm bimod}$, we have three notions of Frobenius-Perron dimension
of objects $Y\in \A_\M^{*\rm op}$: the usual Frobenius-Perron dimension $\FPdim_{\A_\M^{*\rm op}}(Y)$, the Frobenius-Perron 
dimension $\FPdim_\A(Y)$ of the corresponding $A$-bimodule as an object of $\A$, and also 
the Frobenius-Perron dimension of $Y$ as a right $A$-module, regarded as an object of $\M$. 
We will denote the latter one by $\FPdim_\M(Y)$. Since for any $M\in\M,Y\in \A_\M^{*\rm op}$, we have 
$$
\FPdim_\M(M\otimes Y)=\FPdim_\M(M)\FPdim_{\A_\M^{*\rm op}}(Y),
$$ 
we deduce that $\FPdim_\M(Z)\FPdim_{\A_\M^{*\rm op}}(Y)=\FPdim_\M(Z\otimes Y)$ for any 
$Y,Z\in \A_\M^{*\rm op}$. Thus, by the Frobenius-Perron theorem, there exists $\mu>0$ such that 
$\FPdim_\M(Y)=\mu \FPdim_{\A_\M^{*\rm op}}(Y)$ for all $Y\in \A_\M^{*\rm op}$. 
Also, as we showed above, $\FPdim_\A(Y)=\lambda \FPdim_\M(Y)$. 
Altogether, we get $\FPdim_\A(Y)=\lambda\mu \FPdim_{\A_\M^{*\rm op}}(Y)$. Now, plugging in $Y=\bold 1$, we get 
$\FPdim_\A(A)=\lambda\mu$ (since the unit object $\bold 1$ of $\A_\M^{*\rm op}$ is $A$ as an $A$-bimodule). Thus, we get 
\begin{equation}\label{reldim}
\FPdim_\A(Y)=\FPdim_\A(A)\FPdim_{\A_\M^{*\rm op}}(Y),\ Y\in \A_\M^{*\rm op}. 
\end{equation}  

\begin{proposition}\label{algint}  
Let $Y$ be any $A$-bimodule in $\A$. Then $\alpha:=\frac{\FPdim_\A(Y)}{\FPdim_\A(A)}$ is an algebraic integer, and $\alpha\ge 1$. 
\end{proposition} 

\begin{proof}
By \eqref{reldim}, $\alpha=\FPdim_{\A_\M^{*\rm op}}(Y)$, 
which implies both statements.   
\end{proof} 

\subsection{Surjective monoidal functors}\label{behav} 

Let $\mathcal{A}\subseteq \mathcal{B}$ and $\mathcal{C}$ be finite tensor categories over $k$. Let $\mathcal{M}$ be an indecomposable exact 
$\mathcal{A}$-module category. Then $\C\boxtimes \M$ is naturally an exact module category over $\C^{\rm op}\boxtimes \A$. 
Assume that we are given an extension of the action of $\C^{\rm op}\boxtimes \A$ on $\C\boxtimes \M$ to an action of 
$\C^{\rm op}\boxtimes \B$. By Corollary \ref{subcaexact}, $\C\boxtimes \M$ is then an exact $\C^{\rm op}\boxtimes \B$-module category. 
Since the category of right exact $\C^{\rm op}$-linear endofunctors of $\C\boxtimes \M$ 
is $\C\boxtimes \End(\M)$, this extension is encoded in the action functor 
\begin{equation}\label{actionfunctor}
F:\mathcal{B}\to \mathcal{C}\boxtimes \End(\M). 
\end{equation}
It is easy to see that the functor $F$ is monoidal and exact,\footnote{It is clear that if $\M$ is semisimple then $\C\boxtimes \End(\M)$ is 
a multitensor category, and $F$ is a tensor functor.} and moreover for any $M\in \M$ the functor 
$F(?)\otimes (\bold 1\boxtimes M): \B\to \C\boxtimes \M$ is exact. In other words, we can naturally identify 
$\mathcal{C}\boxtimes \End(\M)$ with the category of right exact functors 
from $\M$ to $\C\boxtimes \M$, and upon this identification, 
the functor $F$ lands in the (non-abelian) subcategory of exact functors. 

Let ${\rm Ker}(F)$ be the subcategory of $\mathcal{B}$ mapped to $\End(\M)$ under $F$. 
It is clear that this is a tensor subcategory of $\B$ containing $\A$. 

Let us say that an object of a monoidal category is {\it fully dualizable}
if it admits left and right duals of all orders (see, e.g., \cite{L}). 
Clearly, tensor products and duals of fully dualizable objects are fully dualizable (i.e., fully dualizable objects form a rigid category). It is clear that the functor of tensor product on either side with a fully dualizable object is exact (as it admits a left and right adjoint). 

It is clear that for any $X\in \B$, $F(X)$ is fully dualizable (namely, $F(X)^*=F(X^*)$ and ${}^*F(X)=F({}^*X)$). 

\begin{lemma}\label{projee} If $Z$ is fully dualizable and $Q$ is projective in $\C\boxtimes \End(\M)$, then $Q\otimes Z$ and $Z\otimes Q$ are projective.  
\end{lemma}  

\begin{proof}
Since we have that $\Hom(Q\otimes Z,Y)=\Hom(Q,Y\otimes Z^*)$ and $\Hom(Z\otimes Q,Y)=\Hom(Q,{}^*Z\otimes Y)$, both functors are exact. 
\end{proof} 

The following theorem generalizes \cite[Theorem 2.5]{EO} 
(recovered in the special case $\M=\Vect$). 

\begin{theorem}\label{improj} 
Let $F$ be as in (\ref{actionfunctor}). If $F$ is surjective then 

(i) $F$ maps projective objects to projective ones. 
 
(ii) $\C\boxtimes \M$ is an exact module category over $\B$. 
\end{theorem} 

\begin{proof} (i) The proof mimicks \cite[Subsections 2.5, 2.6]{EO}. 

Let $P_i$ be the indecomposable projectives of $\B$, and let $S_i:=F(P_i)$. Write $S_i=T_i\oplus N_i$, where $T_i$ is projective, and $N_i$ has no projective direct summands. 
Our job is to show that $N_i=0$ for all $i$. 

Let $P_i\otimes P_j=\oplus_r c_{ij}^r P_r$. Then by Lemma \ref{projee}, $N_i\otimes N_j$ contains $\oplus_r c_{ij}^r N_r$ as a direct summand, so 
$(\oplus_i N_i)\otimes N_j$ contains $\oplus_r (\sum_i c_{ij}^r) N_r$ as a direct summand. Since $\sum_i c_{ij}^r>0$ for all $j,r$, 
we see that either $N_i$ are all zero (which is what we want to show), or they  are all nonzero. So let us assume for the sake of contradiction 
that $N_i\ne 0$ for all $i$. 

For any $j$ we have 
$$
(\oplus_i N_i)\otimes N_j\supseteq \bigoplus_r (\sum_i c_{ij}^r) N_r.
$$
Let $X_j$ be the simples of $\B$, and $d_j$ be their Frobenius-Perron dimensions. 
For any $i$, $r$, we have $\sum_j d_jc_{ij}^r=D_id_r$, where $D_i:=\FPdim(P_i)$. 
Thus, multiplying the latter inclusion by $d_j$ and summing over $j$,
we get an entry-wise inequality of matrices (acting on ${\rm Gr}(\C\boxtimes \M)$)
$$
[\oplus_i N_i]\sum_j d_j[N_j]\ge \left(\sum_i D_i\right)\sum_r d_r[N_r].
$$
This implies that the largest eigenvalue of the matrix of $[\oplus_i N_i]$ 
is at least $\sum_i D_i$, which is the same as the largest eigenvalue of $[\oplus_i F(P_i)]$.
By Lemma \ref{lal}, this implies that $N_i=F(P_i)$ for all $i$. Thus, $F(P_i)$ has no projective direct
summands for all $i$. 

However, let $Q$ be an indecomposable projective object in $\C\boxtimes \End(\M)$. 
Then $Q$ is injective by the quasi-Frobenius property. 
Since $F$ is surjective, $Q$ is a subquotient, hence (using Lemma \ref{eqcon}) a direct summand 
of $F(P)$ for some projective $P\in \B$. Hence $Q$ is a direct summand 
of $F(P_i)$ for some $i$, which gives the desired contradiction. 

(ii) By (i), we have 
\begin{equation}\label{fpi}
F(P_i)=\bigoplus_{s,l,r} a_{islr}Q_s\boxtimes R_l\boxtimes R_r^\vee, 
\end{equation} 
where $Q_s,R_l,R_r^\vee$ are the indecomposable projectives of $\C$, $\M$, $\M^\vee$.
Let $Y_t$ be the simples of $\C$, and $M_j$ be the simples of $\M$. Then 
$$F(P_i)\otimes (Y_t\boxtimes M_j)=\bigoplus_{s,l}a_{islj}(Q_s\otimes Y_t)\boxtimes R_l,$$ which is clearly projective. 
This implies (ii). 
\end{proof}

\begin{proposition}\label{regob} The regular object of $\C\boxtimes \M$ as a $\B$-module category is $\alpha \cdot R_\C\boxtimes R_\M$, 
where  $\alpha=\frac{\FPdim(\B)}{\FPdim(\A)\FPdim(\C)}$.
\end{proposition} 

\begin{proof}
The action of $\B$ on $\C\boxtimes\M$ extends to an action of $\C^{\rm op}\boxtimes \B$, so 
the regular object is the same for the two actions, up to scaling. 
But $\C^{\rm op}\boxtimes \B$ has a tensor subcategory $\C^{\rm op}\boxtimes \A$, 
and for the action of this subcategory, the regular object is clearly $R_\C\boxtimes R_\M$. 
This implies the statement (after recalculating normalizations).  
\end{proof} 

\subsection{Behavior of regular objects under surjective monoidal functors}

Retain the setup of Subsection 2.7, and let $F$ be as in (\ref{actionfunctor}).

\begin{theorem}\label{Fofreg}
If $F$ is surjective then 

(i) $F(R_\B)=\alpha\cdot R_\C\boxtimes F(R_\A)=\alpha \cdot R_\C\boxtimes R_\M\boxtimes R_{\M^\vee}$, in \linebreak
$K_0(\C\boxtimes \End(\M))\otimes_{\Bbb Z}\Bbb R$, where  $\alpha=\frac{\FPdim(\B)}{\FPdim(\A)\FPdim(\C)}$.

(ii) $F(R_\B)$ contains as a direct summand $Q\boxtimes F(R_\A)$, where $Q$ is the projective cover of $\bold 1$ in $\C$.
\footnote{Here and below, if $R,S$ are formal linear combinations of projectives in some finite abelian category with nonnegative real coefficients, then we say that 
$R$ contains $S$ as a direct summand if $R-S$ has nonnegative coefficients.} 
 
(iii) $\alpha\ge 1$. 

(iv) Let $M,N\in \M$, and $\underline{\Hom}_\B(\bold 1\boxtimes M,\bold 1\boxtimes N)\in \B$ be the internal Hom from $\bold 1\boxtimes M$ to  
$\bold 1\boxtimes N$ in the $\B$-module category $\C\boxtimes \M$. Then 
$$
\FPdim(\underline{\Hom}_\B(\bold 1\boxtimes M,\bold 1\boxtimes N))=\alpha \FPdim(M)\FPdim(N). 
$$
\end{theorem} 

\begin{proof}
(i) Clearly, $F(R_\B)$ is an eigenvector of left multiplication by $F(X)$, $X\in \B$, with eigenvalue $\FPdim(X)$, 
and under right multiplication by $F(Z)$, $Z\in \A$, with eigenvalue $\FPdim(Z)$. 
Thus the statement follows from Theorem \ref{improj}, Proposition \ref{regob},  and comparisons of largest eigenvalues.  

(ii) For each simple object $X_i$ of $\A$, let $P_i'$ be its projective cover in $\A$, and let $P_i$ be its projective cover in $\B$.  
We have a surjective morphism $\phi_i: P_i\to P_i'$, and hence a surjective morphism $F(\phi_i): F(P_i)\to F(P_i')$.  
But by Theorem \ref{improj}, $F(P_i')$ is a projective object in $\End(\M)$. Thus, 
$F(P_i)$ contains $Q\boxtimes F(P_i')$ as a direct summand. Therefore, setting $d_i:=\FPdim(X_i)$, we get that $F(R_\B)$ contains 
$$\sum_{i: X_i\in \A}d_iQ\boxtimes F(P_i')=Q\boxtimes F(R_\A)$$ as a direct summand. 

(iii) Since $R_\C$ contains $Q$ as a summand with multiplicity $1$, we get from (i) and (ii) 
that $\alpha\ge 1$. 

(iv) We keep the above notation. By \eqref{fpi}, we have
$$
F(P_i)\otimes (\bold 1\boxtimes M)=\bigoplus_{s,l} \left(\sum_r a_{islr}[M:M_r]\right)Q_s\boxtimes R_l.
$$
Hence, for all $i$
\begin{eqnarray*}
\lefteqn{[\underline{\rm Hom}_\B(\bold 1\boxtimes M,\bold 1\boxtimes N):X_i]
}\\
& = & \dim\Hom(P_i,\underline{\rm Hom}_\B(\bold 1\boxtimes M,\bold 1\boxtimes N))\\
& = & 
\dim\Hom(F(P_i)\otimes (\bold 1\boxtimes M),\bold 1\boxtimes N)\\
& = & \sum_{l,r}a_{i0lr}[N:M_l][M:M_r]
\end{eqnarray*}
(where $Q_0=Q$). Multiplying this by $d_i=\FPdim(X_i)$ and summing over $i$, we get 
$$
\FPdim(\underline{\rm Hom}_\B(\bold 1\boxtimes M,\bold 1\boxtimes N))=\sum_{i,l,r}d_ia_{i0lr}[N:M_l][M:M_r]. 
$$
But 
$\sum_i d_ia_{islr}$ is the coefficient of $Q_s\boxtimes R_l\boxtimes R_r^\vee$ in $F(R_\B)$. 
So, by (i), we have 
$$
\sum_i d_ia_{i0lr}=\alpha \FPdim(M_l)\FPdim(M_r).
$$
Thus, we get 
\begin{eqnarray*}
\lefteqn{\FPdim(\underline{\Hom}_\B(\bold 1\boxtimes M,\bold 1\boxtimes N))}\\
& = & \alpha\sum_{l,r}\FPdim(M_l)\FPdim(M_r)[N:M_l][M:M_r]\\
& = & \alpha \FPdim(M)\FPdim(N),
\end{eqnarray*}
as desired.
\end{proof} 

\subsection{Duality} 
Let $\A$, $\B$, $\C$, $\M$ be as above. Let
\begin{equation}\label{ES}
\mathcal{A}\xrightarrow{\iota} \mathcal{B}\xrightarrow{F} \mathcal{C}\boxtimes \End(\M)
\end{equation}
be a sequence of tensor functors, such that $\iota$ is injective, $\A\subseteq \Ker(F)$ (i.e., $F(\A)\subseteq \End(\M)$), and $F$ is surjective. 
Let $\mathcal{N}$ be an indecomposable exact module category over $\mathcal{C}$. 
Note that $\mathcal{N}\boxtimes \mathcal{M}$ is an exact module category over 
$\mathcal{C}\boxtimes \End(\M)$, and $(\mathcal{C}\boxtimes \End(\M))_{\mathcal{N}\boxtimes \mathcal{M}}^*=\mathcal{C}_{\mathcal{N}}^*$.

Consider the dual sequence to (\ref{ES}) with respect to $\N\boxtimes \M$: 
\begin{equation}\label{ES1}
\mathcal{A}_{\mathcal{M}}^*\boxtimes \End(\mathcal{N})\xleftarrow{\iota^*} \mathcal{B}_{\mathcal{N}\boxtimes \mathcal{M}}^*\xleftarrow{F^*} \mathcal{C}_{\mathcal{N}}^*.
\end{equation}

It is clear that $\iota^*$ and $F^*$ are exact monoidal functors. 

\begin{proposition}\label{dualseq} 
(i) $\B_{\N\boxtimes \M}^*$ is a tensor category. 

(ii) $\C_\N^*\subseteq \Ker(\iota^*)$. 

(iii) $F^*$ is injective. 

(iv) $\iota^*$ is surjective.  
\end{proposition} 

\begin{remark} 1. This proposition is a generalization of 
\cite[Theorem 3.46]{EO}, which is recovered in the special cases 
$\C=\Vect$ and $\A=\Vect$. 

2. Note that $\N\boxtimes \M$ is not, in general, an exact $\A$-module category. By the dual category $\A^*_{\N\boxtimes \M}$ we mean the category of 
right exact $\A$-linear endofunctors of $\N\boxtimes \M$, which is easily seen to be the category $\A_\M^*\boxtimes \End(\N)$. 
\end{remark} 

\begin{proof}
(i) By Theorem \ref{improj}(i), $\mathcal{N}\boxtimes \mathcal{M}$ is an indecomposable exact module category over $\mathcal{B}$, so $\mathcal{B}_{\mathcal{N}\boxtimes \mathcal{M}}^*$ is 
a tensor category. 

(ii) Since $F\circ \iota$ is the action map of $\A$ on $\M$ (tensored with the identity functor on $\N$), we see that 
$\iota^*\circ F^*$ is the action map of $\C_\N^*$ on $\N$ (tensored with the identity functor on $\M$). In particular, $\C_\N^*\subseteq {\rm Ker}(\iota^*)$.  

(iii) Recall that $(\C\boxtimes \End(\M))^*_{\N\boxtimes \M}=\C_\N^*$ is the category of functors $G: \N\boxtimes \M\to \N\boxtimes\M$ 
with a $\C\boxtimes\End(\M)$-linear structure, i.e., equipped with isomorphisms $J_{ZL}: Z\otimes G(L)\to G(Z\otimes L)$ functorial in $Z\in \C\boxtimes \End(\M)$, $L\in \N\boxtimes \M$, and satisfying 
appropriate compatibility relations. Any such functor $(G,J)$ in particular defines a $\B$-linear endofuctor of $\N\boxtimes \M$ (which amounts 
to keeping only $J_{ZL}$ for $Z=F(X)$, $X\in \B$). Now, a morphism $\phi: (G_1,J_1)\to (G_2,J_2)$ is a family of morphisms $\phi_L: G_1(L)\to G_2(L)$, 
$L\in \N\boxtimes \M$, which is functorial in $L$ and respects $J_1,J_2$. Therefore, any morphism $\phi$ in $(\C\boxtimes \End(\M))^*_{\N\boxtimes \M}$
defines a morphism $F^*(\phi)$ in $\B_{\N\boxtimes \M}^*$, and the assignment $\phi\mapsto F^*(\phi)$ is injective. Moreover, since $F$ is surjective, any functorial collection of maps 
$\phi_L: G_1(L)\to G_2(L)$ which preserves $J_1,J_2$ for $Z=F(X)$ does so for any $Z\in \C\boxtimes \End(\M)$ (since any such $Z$ is a subobject in $F(X)$ for some $X$).  
This implies that $F^*$ is fully faithful, i.e., injective. 

(iv) Note that we have an action of $\C_\N^*$ on $\N\boxtimes \M$ (in the first component) which commutes 
with the action of $\B$. Thus, $\N\boxtimes \M$ is an exact module category over $\B\boxtimes \C_\N^*$. 
So by \cite[Theorem 3.46]{EO}, we have a surjective tensor functor $$\Xi: (\B\boxtimes \C_\N^*)^*_{\N\boxtimes \M}\to \B^*_{\N\boxtimes \M},$$ and also, 
since $\A\boxtimes \C_\N^*$ is contained in $\B\boxtimes \C_\N^*$, a surjective tensor functor 
$$\Psi: (\B\boxtimes \C_\N^*)^*_{\N\boxtimes \M}\to (\A\boxtimes \C_\N^*)^*_{\N\boxtimes \M}=\A_\M^*\boxtimes \C.$$ Moreover, we have 
$\iota^*\circ \Xi=({\rm Id}\boxtimes \gamma)\circ \Psi$, where $\gamma: \C\to \End(\N)$ is the action functor. 
The functor $\gamma$ is surjective by Proposition \ref{critex2}(ii), so the functor $({\rm Id}\boxtimes \gamma)\circ \Psi$ is surjective. 
Hence, the functor $\iota^*\circ \Xi$ is surjective. In other words, any object $Z\in \A_\M^*\boxtimes \End(\N)$ 
is a subquotient of $\iota^*(\Xi(X))$ for some $X$. This implies that $\iota^*$ is surjective. 
\end{proof} 

\begin{proposition} \label{bamexact} Under the above assumptions, 
$\B\boxtimes_\A \M$ is an exact $\B$-module category.  
\end{proposition} 

\begin{proof} By Proposition \ref{dualseq}(i) and \cite[Theorem 3.31]{EO}, it suffices to show that the $\B_{\C\boxtimes\M}^*$-module category
${\rm Fun}_\B(\B\boxtimes_\A\M,\C\boxtimes\M)$ is exact. But 
$$
{\rm Fun}_\B(\B\boxtimes_\A\M,\C\boxtimes\M)={\rm Fun}_\A(\M,\C\boxtimes\M)=\C\boxtimes \A_\M^*. 
$$ 
The action of $\B_{\C\boxtimes\M}^*$ on $\C\boxtimes \A_\M^*$ is via the monoidal functor 
$\iota^*: \B_{\C\boxtimes \M}^*\to \A_\M^*\boxtimes \End(\C)$. 
By Proposition \ref{dualseq}(iv), the functor $\iota^*$ is surjective.
By Theorem \ref{improj}(ii), this implies that the module category $\C\boxtimes \A_\M^*$ over $\B^*_{\C\boxtimes \M}$ 
is exact.  
\end{proof} 

\begin{question} Let $\A\subseteq \B$ be any finite tensor categories, and $\M$ be an exact $\A$-module 
category. Is $\B\boxtimes_\A\M$ always an exact $\B$-module category?   
\end{question} 

\begin{proposition}\label{algint1} 
The number $\alpha$ in Theorem \ref{Fofreg} is an algebraic integer. 
\end{proposition}

\begin{proof} Pick a nonzero object $M\in \M$, and consider the algebra $A:=\underline{\End}_\A(M)$ in $\A$. 
Then $\M={\rm mod}-A_{\A}$, and $\B\boxtimes_\A \M={\rm mod}-A_{\B}$ (the categories of right $A$-modules in $\A$ and $\B$, respectively). 
On the other hand, consider the algebra $A':=\underline{\End}_\B(\bold 1\boxtimes M)$
in $\B$. Then $\C\boxtimes \M={\rm mod}-A'_{\B}$. 
We have a natural morphism $\Phi: A\to A'$, which is the image of the right action map 
$a: M\otimes A\to M$ under the isomorphism  
$$
\Hom_\M(M\otimes A,M)\cong \Hom_\B(A,A').
$$
Thus, $A'$ is an $A$-bimodule in $\B$, i.e., an object of the category 
$\B_{\B\boxtimes_\A\M}^{*\rm op}$. 

By Theorem \ref{Fofreg}(iv), we have $\FPdim(A)=\FPdim(M)^2$, while $\FPdim(A')=\alpha \FPdim(M)^2$, where 
$\alpha=\frac{\FPdim(\B)}{\FPdim(\A)\FPdim(\C)}$ (because of the normalization of the Frobenius-Perron 
dimensions in $\C\boxtimes \M$ as a $\B$-module). Hence, $\alpha=\frac{\FPdim(A')}{\FPdim(A)}$. 
But $A'$ is an $A$-bimodule in $\B$.Thus, the statement follows from 
Proposition \ref{algint} and 
Proposition \ref{bamexact}.
\footnote{This proof generalizes an argument in the semisimple case due to V. Ostrik. Note that this also gives another proof of Theorem \ref{Fofreg}(iii).} 
\end{proof} 

\section{Exact sequences} 

\subsection{The definition of an exact sequence of tensor categories with respect to a module category} 

Let $F:\B\to \C\boxtimes \End(\M)$ be a functor as in (\ref{actionfunctor}). 

\begin{definition}\label{gesd}  
Assume that $F$ is a surjective (= dominant) functor such that
$\mathcal{A}=\Ker(F)$.
We say that $F$ is {\em normal}, or {\em defines an exact sequence with respect to $\mathcal{M}$}
$$
\mathcal{A}\xrightarrow{\iota} \mathcal{B}\xrightarrow{F} \mathcal{C}\boxtimes \End(\M),
$$
if one of two equivalent conditions holds: 

(i) For any $X\in \mathcal{B}$ there exists a subobject $X_0\subseteq X$ such that $F(X_0)$ 
is the largest subobject of $F(X)$ contained in $\End(\M)\subseteq \C\boxtimes \End(\M)$.  

(ii) For any $X\in \mathcal{B}$ there exists a quotient object $X_0$ of $X$ such that $F(X_0)$ 
is the largest quotient object of $F(X)$ contained in $\End(\M)\subseteq \C\boxtimes \End(\M)$.

In this case we will also say that $\B$ is an {\it extension of $\C$ by $\A$ 
with respect to $\M$}. 
\end{definition}

\begin{remark} The equivalence of conditions (i) and (ii) follows by taking duals.  
\end{remark}

Note that if $\M={\rm Vec}$, this definition coincides with that of \cite{BN}. 
In particular, if $H\subseteq G$ are finite groups, $\B:=\Rep(G)$, $\C:=\Rep(H)$,
$F$ is the restriction functor, $\A:={\rm Ker}(F)$, and 
$\M:={\rm Vec}$, then $F$ is normal if and only if $H$ is a normal subgroup of $G$, which motivates the terminology. 
Also, it is clear that if $\B=\C\boxtimes \A$ and $F$ is the obvious functor, 
then $F$ defines an exact sequence with respect to $\M$. 
So $\C\boxtimes \A$ is an extension of $\C$ by $\A$ with respect to any 
indecomposable exact $\A$-module category $\M$ (e.g., $\M=\A$). 

\begin{lemma}\label{injprojsuff}
(i) Suppose condition (i) in Definition \ref{gesd} holds for injective (= projective) objects $X$. Then $F$ is normal. 

(ii)  Suppose condition (ii) in Definition \ref{gesd} holds for projective \linebreak (= injective) objects $X$. Then $F$ is normal.

(iii) In (i) and (ii), it suffices to restrict to indecomposable objects $X$. 
\end{lemma} 

\begin{proof}  
(i) Let $X\in \B$, and $L$ be the sum of all subobjects of $F(X)$ contained in 
$\End(\M)$. Let $I(X)\supseteq X$ be the injective hull of $X$. Let $S$ 
be  the sum of all objects in $F(I(X))$ contained in $\End(\M)$. 
Then by condition (i) for injectives, $S=F(Z)$ for some $Z\in \A$, $Z\subseteq I(X)$, and $F(Z)\supseteq L$. 
Thus $F(Z)\cap F(X)=F(Z\cap X)\supseteq L$. But $Z\cap X\in \A$, since $Z\in \A$, so 
$F(Z\cap X)\in \End(\M)$, and hence $L=F(Z\cap X)$, as desired. 

(ii) is obtained from (i) by taking duals.

(iii) Let $J(Z)$ be the sum of all subobjects of $Z\in \C\boxtimes \End(\M)$ which are contained in $\End(\M)$. 
Then $J(\oplus_{i=1}^n Z_i)=\oplus_{i=1}^nJ(Z_i)$. Indeed, it is clear that $J(\oplus_{i=1}^n Z_i)\supseteq\oplus_{i=1}^nJ(Z_i)$. 
On the other hand, the projection of $J(\oplus_{i=1}^n Z_i)$ to each $Z_i$ is clearly contained in $J(Z_i)$, i.e.,
 $J(\oplus_{i=1}^n Z_i)\subseteq \oplus_{i=1}^nJ(Z_i)$, as claimed.  This implies that in (i), we can restrict to 
 indecomposable $X$. The proof that in (ii) we can restrict to indecomposable $X$ is similar. 
\end{proof} 

\subsection{Characterization of exact sequences in terms of Frobenius-Perron dimensions}

The following theorem gives an equivalent definition of an exact sequence. 
Let $F:\B\to \C\boxtimes \End(\M)$ be as in (\ref{actionfunctor}). 

\begin{theorem}\label{eds} Assume that $F$ is surjective, and let $\alpha=\frac{\FPdim(\B)}{\FPdim(\A)\FPdim(\C)}$. Then the following are equivalent: 

(i) $\alpha=1$, i.e., $\FPdim(\B)=\FPdim(\A)\FPdim(\C)$.

(ii) $\A={\rm Ker}(F)$ and $F$ is normal (i.e., $F$ defines an exact sequence of tensor categories). 
\end{theorem} 

\begin{proof}
Let us show that (i) implies (ii). First of all, we can replace $\A$ with $\A':={\rm Ker}(F)\supseteq \A$, 
in which case $\alpha$ will be replaced by 
$$
\alpha':=\frac{\FPdim(\B)}{\FPdim(\A')\FPdim(\C)}\le \alpha.
$$ 
Since $\alpha=1$ and $\alpha'\ge 1$ by Theorem \ref{Fofreg}(iii), we see that $\alpha'=\alpha=1$, 
and hence $\FPdim(\A)=\FPdim(\A')$. Thus, $\A=\A'={\rm Ker}(F)$. 

Now, by Theorem \ref{Fofreg}(ii), $F(R_\B)=Q\boxtimes F(R_\A)\oplus T$ for some $T$. Since $\alpha=1$, 
$T$ has no direct summand of the form $Q\boxtimes R_l\boxtimes R_r^\vee$ with positive coefficients 
(where $R_j$, $R_r^\vee$ are the indecomposable projectives in $\M$, $\M^\vee$). In other words, for any $i$ such that 
$X_i\in \B$ but $X_i\notin \A$, and any $Z\in \End(\M)$, we have $\Hom(F(P_i),\bold 1\boxtimes Z)=0$
(where $P_i$ is the projective cover of $X_i$). This means that $F(P_i)$ has no nonzero quotients contained in $\End(\M)$. 

On the other hand, if $X_i\in \A$ and $P_i'$ is its projective cover in $\A$, it follows similarly that 
$$
\Hom(F(P_i),\bold 1\boxtimes Z)=\Hom(F(P_i'),Z).
$$
Now let $L_i$ be the maximal quotient of $F(P_i)$ in $\End(\M)$. Then $\Hom(F(P_i),\bold 1\boxtimes Z)=\Hom(L_i,Z)$. 
This means that the (a priori injective) map $\Hom(F(P_i'),Z)\to \Hom(L_i,Z)$  is in fact an isomorphism for any $Z$, so $L_i=F(P_i')$. 

By Lemma \ref{injprojsuff}(ii),(iii), this implies that $F$ is normal. 

The proof that (ii) implies (i) is obtained by running the above proof in reverse. 
First consider $X_i\in \B$ such that $X_i\notin \A={\rm Ker}(F)$. Then $P_i$ has no nonzero quotients 
contained in $\A$ (as any such quotient projects to $X_i$). Since $F$ is normal, this implies that 
$F(P_i)$ has no nonzero quotients contained in $\End(\M)$, hence $\Hom(F(P_i),\bold 1 \boxtimes Z)=0$ for all $Z\in \End(\M)$. 
So $F(P_i)$ has no summands $Q\boxtimes R_l\boxtimes R_r^\vee$ with positive coefficients. Similarly, if $X_i\in \A$, then 
the maximal quotient of $P_i$ in $\A={\rm Ker}(F)$ is the projective cover $P_i'$ of $X_i$ in $\A$, so 
by normality of $F$, the maximal quotient of $F(P_i)$ in $\End(\M)$ is $F(P_i')$. 
Hence $\Hom(F(P_i),\bold 1\boxtimes Z)=\Hom(F(P_i'),Z)$ for all $Z\in \End(\M)$. 
Thus, $F(P_i)=Q\boxtimes F(P_i')\oplus T_i$, where $T_i$ has no summands $Q\boxtimes R_l\boxtimes R_r^\vee$ with positive coefficients.
Hence, $F(R_\B)=Q\boxtimes F(R_\A)\oplus T$, where $T$ has no summands $Q\boxtimes R_l\boxtimes R_r^\vee$ with positive coefficients. 
This implies that $\alpha=1$. 
\end{proof} 

\begin{remark} 
In the case $\M=\Vect$ for fusion categories,
Theorem \ref{Fofreg} and Theorem \ref{eds} reduce to 
\cite[Proposition 3.10]{BN}.
\end{remark} 

\subsection{Characterization of exact sequences in terms of Tambara tensor products}

\begin{theorem}\label{ges}
Let
$F:\mathcal{B}\to \mathcal{C}\boxtimes \End(\M)$ be an exact monoidal functor as in (\ref{actionfunctor}). The following are equivalent:

(i) $F$ is surjective, $\mathcal{A}={\rm Ker}(F)$ and $F$ is normal, i.e., $$\mathcal{A}\xrightarrow{\iota} \mathcal{B}\xrightarrow{F} \mathcal{C}\boxtimes \End(\M)$$ is an exact sequence with respect to $\mathcal{M}$.

(ii) The natural functor $\Phi_*: \mathcal{B}\boxtimes_{\mathcal{A}}\mathcal{M}\to \mathcal{C}\boxtimes \mathcal{M}$, given by 
$$
\mathcal{B}\boxtimes_{\mathcal{A}}\mathcal{M}\xrightarrow{F\boxtimes_{\mathcal{A}} \id_{\mathcal{M}}} \mathcal{C}\boxtimes \End(\M)\boxtimes_{\mathcal{A}}\mathcal{M}=\mathcal{C}\boxtimes \mathcal{M}\boxtimes \mathcal{A}_{\mathcal{M}}^{*\rm op}\xrightarrow{\id_{\mathcal{C}}\boxtimes \rho} \mathcal{C}\boxtimes \mathcal{M},
$$ 
is an equivalence 
(where $\rho:\mathcal{M}\boxtimes \mathcal{A}_{\mathcal{M}}^{*\rm op} \to \mathcal{M}$ is the right action of $\mathcal{A}_{\mathcal{M}}^{*\rm op}$ on $\mathcal{M}$).

(iii) The natural functor $\Phi^*: {\rm Fun}_{\mathcal{A}}(\mathcal{M},\mathcal{B})\to {\rm Fun} (\mathcal{M},\mathcal{C})$, given by 

$$
{\rm Fun}_{\mathcal{A}}(\mathcal{M},\mathcal{B})\xrightarrow{F\circ ?} {\rm Fun}_{\mathcal{A}}(\mathcal{M},\mathcal{C}\boxtimes \End(\M))=
\mathcal{C}\boxtimes {}^\vee\mathcal{M}\boxtimes_{\mathcal{A}} \End(\M)=
$$
$$
\C\boxtimes \A_\M^*\boxtimes \M^\vee=\C\boxtimes \M^\vee\boxtimes \A_\M^*\xrightarrow{\id_{\mathcal{C}}\boxtimes \rho^{\rm op}} \mathcal{C}\boxtimes \mathcal{M}^\vee={\rm Fun} (\mathcal{M},\mathcal{C}),
$$
is an equivalence (where ${\rm Fun} (\mathcal{M},\mathcal{C})$ is the category of right exact functors from $\M$ to $\C$).
\end{theorem}

\begin{proof}
The equivalence of (ii) and (iii) follows from (\ref{tenprod}).

Let us prove that (ii) implies (i). First let us show that $F$ is surjective. We keep the above notation. 
By (ii), for any $s$ and $l$, the object $Q_s\boxtimes R_l$ is a subquotient (in fact, a direct summand) 
in $F(X)(\bold 1\boxtimes M)$ for some $X\in \B$ and $M\in \M$. 
By Proposition \ref{critex2}(i), by replacing $X$ by $X\otimes Z$, $Z\in \A$, we may assume that 
$M=M_r$ for any given $r$. Since for any $M\in \M$, $\Hom(M,M_r)^*\otimes M_r$ is canonically a quotient of $M$, this means 
that $F(X)$ has $Q_s\boxtimes R_l\boxtimes M_r^\vee$ as a quotient (where $M_r^\vee$ stands for $M_r$ regarded as an object of $\M^\vee$). 
Let $P_0'$ be the projective cover of $\bold 1$ in $\A$. Then $P_0'\otimes M_r$ contains $R_r$ as a direct summand (since $\M$ is an exact $\A$-module). 
Hence, we see that $F(X\otimes (P_0')^*)$ involves $Q_s\boxtimes R_l\boxtimes R_r^\vee$ as a quotient, hence as a direct summand. 
Since $s,l,r$ can be arbitrary, we conclude that $F$ is surjective. 

To prove the rest of the statements, consider the algebras $A,A'$ defined in the proof of Proposition \ref{algint1}, and recall that we have 
a natural homomorphism $\Phi: A\to A'$. It is easy to see that $\Phi$ is injective, i.e., $A\subseteq A'$. 
It is also easy to show that the natural functor $\Phi_*: B\boxtimes_\A \M\to \C\boxtimes \M$ is the induction functor for
$\Phi$. Thus, if $\Phi_*$ is an equivalence, then $\Phi$ is an isomorphism. 
Hence, $\alpha=1$. But then by Theorem \ref{eds}, 
$\A={\rm Ker}(F)$ and $F$ is normal.     

The proof that (i) implies (ii) is obtained by reversing the above argument. 
Namely, by Theorem \ref{eds}, we have $\alpha=1$. Hence, 
$\Phi$ is an isomorphism. This implies (ii). 
\end{proof}

Specializing to the case $\mathcal{M}=\Vect$, we get the following special case of Theorem \ref{ges}.

\begin{corollary}\label{bnes}
Let $\mathcal{A}\xrightarrow{\iota} \mathcal{B}\xrightarrow{F} \mathcal{C}$ be a sequence of tensor functors between finite tensor categories, such that 
$\iota$ is fully faithful, and the composition $F\circ \iota$ lands in $\Vect\subseteq \C$. Then the following conditions are equivalent:

(i) The sequence is exact in the sense of \cite{BN}.

(ii) The natural functor $\mathcal{B}\boxtimes_{\mathcal{A}}\Vect \to \mathcal{C}$, given by $$\mathcal{B}\boxtimes_{\mathcal{A}}\Vect \to \mathcal{C}\boxtimes_{\mathcal{A}}\Vect=\mathcal{C}\boxtimes \mathcal{A}_{\Vect}^{*\rm op}\to \mathcal{C},$$ is an equivalence. 

(iii) The natural functor ${\rm Fun}_{\mathcal{A}}(\Vect,\mathcal{B})\to \mathcal{C}$, given by the diagram
$${\rm Fun}_{\mathcal{A}}(\Vect,\mathcal{B})\to {\rm Fun}_{\mathcal{A}}(\Vect,\mathcal{C})\to {\rm Fun}(\Vect,\mathcal{C})=\mathcal{C},$$ is an equivalence. \qed
\end{corollary}

\subsection{Semisimplicity of extensions of semisimple categories} 

Next we prove the following theorem, which extends \cite[Corollary 4.16]{BN}.

\begin{theorem}\label{extfc}
Let $\mathcal{A}$, $\mathcal{C}$ be fusion categories, and let $\mathcal{B}$ be a finite tensor category. Suppose that $\mathcal{M}$ is an indecomposable exact (i.e., semisimple) module category over $\mathcal{A}$ such that there is an exact sequence 
$$\mathcal{A}\xrightarrow{\iota} \mathcal{B}\xrightarrow{F} \mathcal{C}\boxtimes \End(\M)
$$
with respect to $\mathcal{M}$. Then $\mathcal{B}$ is a fusion category (i.e., it is semisimple).
\end{theorem}

\begin{proof}
The category $\B\boxtimes_\A\M$ is a right module category over the finite tensor category $\A_\M^{*\rm op}$.
This module category is semisimple, since by Theorem \ref{ges}, it is equivalent as a $\B$-module category 
(in particular, just as an abelian category) to $\C\boxtimes \M$, and both $\C$ and $\M$ are semisimple. 
Therefore, $\B\boxtimes_\A\M$ is an exact (possibly decomposable) right module category over $A_\M^{*\rm op}$. 
Hence, by \cite[Theorem 3.31]{EO}, the category ${\rm Fun}_{A_\M^{*\rm op}}(\B\boxtimes_\A \M,\M)$ is an exact right $\A$-module category. 
But by  \eqref{tenprod}, we have 
\begin{eqnarray*}
\lefteqn{{\rm Fun}_{A_\M^{*\rm op}}(\B\boxtimes_\A \M,\M)=\M\boxtimes_{\A_\M^{*\rm op}}(\B\boxtimes_\A \M)^\vee}\\
& = &
\M\boxtimes_{\A_\M^{*\rm op}}(\M^\vee\boxtimes_\A\B^\vee)= 
(\M\boxtimes_{\A_\M^{*\rm op}}\M^\vee)\boxtimes_\A\B^\vee\\
& = & 
\A^\vee\boxtimes_\A \B^\vee=(\B\boxtimes_\A \A)^\vee=\B^\vee.
\end{eqnarray*}
Since $\A$ is a fusion category, this implies that $\B$ is also a fusion category, as desired. 
\end{proof} 

\section{Dualization}

\subsection{Duality of exact sequences} 
In this section, we show that our notion of exact sequences is self-dual. 
Suppose we have a sequence of the form \eqref{ES} which is an exact sequence of tensor categories with respect to $\M$.  
Let $\N$ be an indecomposable exact $\C$-module category. 

\begin{theorem}\label{duales}
The dual sequence (\ref{ES1}) of (\ref{ES}) with respect to $\mathcal{N}\boxtimes \mathcal{M}$
is an exact sequence with respect to the indecomposable exact module category $\mathcal{N}$ over $\mathcal{C}_{\mathcal{N}}^*$.
Moreover, taking the double dual is the identity operation, i.e., the dual of the sequence (\ref{ES1}) with respect to $\M\boxtimes \N$ is the original sequence (\ref{ES}).
\end{theorem}

\begin{proof} By Proposition \ref{dualseq}, $\B_{\N\boxtimes\M}$ is a tensor category, 
$\C_\N^*\subseteq \Ker(\iota^*)$, and $F^*,\iota^*$ are exact monoidal functors such that $F^*$ is 
injective and $\iota^*$ is surjective. Now, by \cite[Corollary 3.43]{EO}, we have 
$$
\FPdim(\A_\M^*)=\FPdim(\A),\ \FPdim(\C_\N^*)=\FPdim(\C),
$$
and
$$
 \FPdim(\B_{\M\boxtimes \N}^*)=\FPdim(\B).
$$
Hence, by Theorem \ref{eds}, we have 
$$
\FPdim(\B_{\M\boxtimes \N}^*)=\FPdim(\C_\N^*)\FPdim(\A_\M^*).
$$ 
Therefore, using Theorem \ref{eds} again, we conclude that $\C_\N^*={\rm Ker}(\iota^*)$ and $\iota^*$ is normal. 
Thus, (\ref{ES1}) is an exact sequence. 

The last statement (involutivity of duality for exact sequences) is obvious. 
\end{proof}

\subsection{Classes of extensions} 
Let $\A$ and $\C$ be finite tensor categories, and let $\M$ be an 
indecomposable exact $\A$-module category. 
Denote by $\Ext(\C,\A,\M)$ the set of equivalence classes 
of exact sequences (\ref{ES}) with various $\B\supseteq \A$ 
and various functors $F$. Let $\Ext(\C,\A)$ be the 
disjoint union of $\Ext(\C,\A,\M)$ over all equivalence classes of 
$\M$. Given an equivalence class of exact sequences $\mathcal{S}\in \Ext(\C,\A,\M)$, and an indecomposable 
exact $\C$-module category $\N$, denote the equivalence class of the dual exact sequence defined in the previous 
subsection by $\mathcal{S}_\N^*$; we have $\mathcal{S}_\N^*\in \Ext(\A_\M^*,\C_\N^*,\N)$. 
Now, let $\M'$ be an indecomposable exact $\A_\M^*$-module category. 
Then we have $({\mathcal S}_\N^*)_{\M'}^*\in \Ext(\C,\A',\M')$, where 
$\A':=(\A_\M^*)_{\M'}^*$. Thus, we obtain the following result. 

\begin{proposition}\label{extols} 
(i) For any weak Morita equivalence from $\A$ to $\A'$ that maps $\M$ to $\M'$
(see \cite[Section 3.3]{EO}), we have a natural bijection of sets $\Ext(\C,\A,\M)\cong \Ext(\C,\A',\M')$.
In particular, we have a natural bijection of sets $\Ext(\C,\A)\cong \Ext(\C,\A')$ for any 
weak Morita equivalence between $\A$ and $\A'$ (i.e., $\Ext(\C,\A)$ is a Morita invariant of $\A$).  

(ii) 
$\Ext(\C,\A,\M)=\Ext(\C,\A_\M^{*\rm op},\A_\M^{*\rm op})$. \qed
\end{proposition} 

\subsection{Examples of extensions} 
Let us show how one can obtain new exact sequences by dualizing some obvious ones.

\begin{example} For a finite group $G$ and a $3-$cocycle $\omega\in Z^3(G,k^\times)$, let $\Vect(G,\omega)$ be the 
fusion category of finite dimensional $G$-graded vector spaces with associativity defined by $\omega$. Let $\Vect(G):=\Vect(G,1)$. 
If $H\subseteq G$ is a subgroup, and $\psi: H\times H\to k^\times$ is a cochain such that $d\psi=\omega|_{H}$, let
$\mathcal{M}(H,\psi)$ be the corresponding indecomposable module category over $\Vect(G,\omega)$ (see \cite[8.8]{ENO}).
Namely, $\M(H,\psi)$ is the category of right $H$-equivariant finite dimensional $G$-graded vector spaces with associativity defined by $\psi$.
Let $\C(G,H,\omega,\psi)$ be the corresponding opposite dual category, i.e., the category of $H$-biequivariant finite dimensional $G$-graded vector spaces with associativity defined by $\psi$.
E.g., $\C(G,1,\omega,1)=\Vect(G,\omega)$ and $\C(G,G,1,1)=\Rep(G)$. 

Let $$1\to G_1\to G\xrightarrow{f} G_2\to 1$$ be a short exact sequence of finite groups, and let $\omega_2\in Z^3(G_2,k^\times)$. Clearly,
$$
\Vect(G_1)\to\Vect(G,\omega_2\circ f)\xrightarrow{F}\Vect(G_2,\omega_2)
$$   
is an exact sequence with respect to $\mathcal{M}:=\Vect$, where $F$ is the tensor functor corresponding to the homomorphism $f$
(this is an exact sequence in the sense of \cite{BN}).

Let $\N=\M(H_2,\psi_2)$ for some subgroup $H_2\subseteq G_2$ and cochain $\psi_2$ on $H_2$ such that $d\psi_2=\omega_2|_{H_2}$. 
By Theorem \ref{duales}, dualizing the above exact sequence with respect to $\mathcal{N}$ yields an exact sequence
$$
\Rep(G_1)\boxtimes \End(\mathcal{N})\leftarrow \mathcal{C}(G,f^{-1}(H_2),\omega_2\circ f,\psi_2\circ f)^{\rm op}\leftarrow \mathcal{C}(G_2,H_2,\omega_2,\psi_2)^{\rm op}
$$
with respect to $\mathcal{M}(H_2,\psi_2)$.

For $\omega_2=1$ and $\mathcal{N}=\Vect$ (i.e, $H_2=G_2, \psi_2=1$), the resulting exact sequence with respect to $\Vect$ is $$
\Rep(G_1)\leftarrow \Rep(G)\leftarrow \Rep(G_2).
$$

For $\mathcal{N}=\Vect(G_2)$ (i.e, $H_2=1, \psi_2=1$), the resulting exact sequence with respect to $\Vect(G_2)$ is $$
\Rep(G_1)\boxtimes \End (\Vect(G_2))\leftarrow \mathcal{C}(G,G_1,\omega_2\circ f,1)\leftarrow \Vect(G_2,\omega_2).
$$
\end{example}

\begin{example}
Let $\mathcal{C}$ be a finite tensor category acted upon by a finite group $G$. Let $\mathcal{C}^G$ be the equivariantization category (see, e.g., \cite[4.15]{EGNO}), and consider the diagram
$$
\Rep(G)\rightarrow \mathcal{C}^G\xrightarrow{F} \mathcal{C}
,$$ where $F$ is the forgetful functor. By \cite{BN}, it is an exact sequence with respect to $\mathcal{M}:=\Vect$. Since $(\mathcal{C}^G)^*_{\mathcal{C}}=(\C\rtimes G)^{\rm op}$, dualizing it with respect to $\mathcal{N}:=\mathcal{C}$ yields the exact sequence
$$
\Vect(G)\boxtimes \End(\mathcal{C})\leftarrow (\C\rtimes G)^{\rm op}\leftarrow \mathcal{C}^{\rm op}$$ with respect to $\mathcal{C}$. 
\end{example}


\begin{thebibliography}{ABCD}

\bibitem
[BN1]{BN} A. Brugui\`eres and S. Natale. Exact sequences of tensor categories. {\em Int. Math. Res. Not.} {\bf no. 24} (2011), 5644--5705.

\bibitem
[BN2]{BN1} A. Brugui\`eres and S. Natale. Central Exact Sequences of Tensor Categories, Equivariantization and Applications. {\em arXiv:1112.3135}. 

\bibitem[DSS1]{DSS} C. Douglas, C. Schommer-Pries and N. Snyder. 
Dualizable tensor categories. {\em arXiv:1312.7188}.

\bibitem[DSS2]{DSS2} C. Douglas, C. Schommer-Pries and N. Snyder. 
The balanced tensor product of module categories. {\em arXiv:1406.4204}.

\bibitem
[EGNO]{EGNO} P. Etingof, S. Gelaki, D. Nikshych and V. Ostrik. Tensor categories. {\em AMS Mathematical Surveys and Monographs book series}, to appear.
    
\bibitem
[ENO1]{ENO} P. Etingof, D. Nikshych and V. Ostrik. On fusion categories. \textit{Annals of Mathematics} \textbf{162} (2005), 581--642. 

\bibitem
[ENO2]{ENO2} P. Etingof, D. Nikshych and V. Ostrik.
Weakly Group-theretical and solvable fusion categories. 
{\em Advances in Mathematics} {\bf 226} (2011), no. 1, 176--205.

\bibitem
[EO]{EO} P. Etingof and V. Ostrik. Finite tensor categories. \textit{Moscow Math.\ Journal} \textbf{4} (2004), 627--654.

\bibitem[L]{L}
J. Lurie. On the Classification of Topological Field Theories. {\em Current Developments in Mathematics} {\bf 2008} (2009), 129--280. (arXiv:0905.0465.)

\bibitem
[M]{Mo} S. Montgomery. Hopf algebras and their actions on rings. \textit{CBMS Regional Conference Series in Mathematics} \textbf{82} (1993).


\bibitem
[R]{R} D. Radford. Hopf algebras. \textit{World Scientific} (2011).

\bibitem
[S]{Sch} H-J. Schneider. Some remarks on exact sequences of quantum groups. {\em Comm. Alg.} {\bf 21(9)} (1993), 3337--3357.

\bibitem[T]{Ta} D. Tambara.
A duality for modules over monoidal categories of representations
of semisimple Hopf algebras. {\em J. Algebra} {\bf 241}
(2001), 515--547. 

\end{thebibliography}
\end{document}